\documentclass[reqno, 12pt]{amsart} 

\usepackage{amssymb}
\usepackage{amsthm}
\usepackage{amsfonts}
\usepackage{amsmath}
\usepackage{hyperref}
\hypersetup{
   colorlinks=true,
   linkcolor=blue,
}
\usepackage{enumerate}
 \usepackage{mathtools}
\usepackage{tikz}
\usepackage{mathdots}
 \usepackage{yhmath}
\usepackage{cancel}
\usepackage{ulem,xpatch}
\usepackage{color}
\usepackage{bbm}

\newtheorem{lemma}{Lemma}[section]
\newtheorem{proposition}{Proposition}[section]

\newtheorem{condition}{Condition}[section]

\newtheorem{theorem}{Theorem}[section]

\newtheorem{remark}{Remark}[section]

\newcommand{\PP}{\mathbb{P}}
\newcommand{\EE}{\mathbb{E}}
\newcommand{\ZZ}{\mathbb{Z}}
\newcommand{\FF}{\mathcal{F}}
\newcommand{\Rr}{\mathcal{R}}

\newcommand{\Name}{{\rm GERW}}

\makeatletter
\@addtoreset{equation}{section}
\makeatother

\renewcommand\thetable{\thesection.\@arabic\c@table}

\title[generalized excited random walk]{A note on transience of generalized multi-dimensional excited random walks}

\author{Rodrigo B. Alves$^1$, Giulio Iacobelli$^2$, Glauco Valle$^3$}
\thanks{1. Supported by CAPES}
\thanks{3. Supported by CNPq grant 307938/2022-0 and FAPERJ grant E-26/202.636/2019.}

\address{
\newline
\newline
FGV - Escola de matem\'atica aplicada.
\newline Caixa postal 22250-900, Rio de Janeiro, Brasil
\newline
$^1$ e-mail: {\rm \texttt{rodrigo.alves@fgv.br}}
}

\address{
\newline
\newline
UFRJ - Departamento de m\'etodos estat\'{\i}sticos do Instituto de Matem\'atica.
\newline  Caixa Postal 68530, 21945-970, Rio de Janeiro, Brasil
\newline
$^2$ e-mail: {\rm \texttt{giulio@im.ufrj.br}}
\newline
$^3$ e-mail: {\rm \texttt{glauco.valle@im.ufrj.br}}
}

\subjclass[2020]{60K37}
\keywords{excited random walks, non-Markovian processes,  transience}

\begin{document}

\maketitle

\begin{abstract}
We consider a variant of the Generalized Excited Random Walk (\Name) in dimension $d\ge 2$ where the lower bound on the drift for excited jumps is time-dependent and decays to zero.  We show that if the lower bound decays slower than $n^{-\beta_0}$ ($n$ is time),  where $\beta_0$ depends on the transitions of the process, the \Name{} is transient in the direction of the drift.
\end{abstract}



\setcounter{tocdepth}{2}




\section{Introduction}

The multi-dimensional Generalized Excited Random Walk (\Name) was introduced by Menshikov et al. in~\cite{menshikov2012general} following a series of works on multi-dimensional excited random walk \cite{benjamini2003excited,kozma2003excited,berard2007central}. The model considered in ~\cite{menshikov2012general} is a uniformly elliptic random walk with bounded jumps in dimension $d\ge 2$ such that on already visited sites it behaves as a $d$-dimensional martingale with bounded jumps and zero-mean vector and whenever a site is visited for the first time its increment has a drift in some fixed direction $\ell$ of the unit sphere in $\mathbb{R}^d$. They show that the GERW with a drift condition in direction $\ell$ is ballistic in that direction. Besides that, they proved a LLG and a CLT (both for dimensions $d \geq 2$) under stronger hypothesis on the definition of GERW; these particular models were called \textit{excited random walk in random environment}. 

What makes the GERW  an interesting model is the self-interaction encoded in the different behaviors the process has on sites visited for the first time as compared to sites already visited. This makes it an important toy model of non-Markovian random walks. 
Similar works worth mentioning along these lines are~\cite{angel2021balanced,benjamini2011balanced,peres2016martingale}.  
A natural question is  what happens to \Name{} when the strength of the drift on the first visits decreases with time. Would the process still be ballistic in the direction of the drift? What about LLN and CLT?  We propose here a variation of \Name{} to contemplate this case by assuming that the drift in a fixed direction $\ell$ at time $n$ is of order $n^{-\beta}$ if at this time  a site is visited for the first time. If $\{X_n\}_{n\geq 0}$ denotes the \Name{} under this weaker condition on the directional drift,   we show that  $\lim_{n \rightarrow \infty} X_n \cdot \ell = \infty$ with positive probability ({\it directional transience})  if $\beta$ is sufficiently small. This shows that our model has an intermediary behavior between a mean-zero random walk (non-excited) and the \Name{}  considered in \cite{menshikov2012general} (see, Remark~\ref{rem:sub-balistic}).  In \cite{AIV} we discuss limit theorems for $\beta \ge 1/2$ for a particular class of \Name{} in the same spirit of the excited random walk in a random environment discussed in 
\cite{menshikov2012general}.

Our proof is based on an adaptation of the arguments presented in~\cite{menshikov2012general} for our time-inhomogeneous case. It involves the use of a bound on the range of the walk to guarantee that the walk receives enough impulse in the drift direction. 


\subsection{Definition of the model and main result}\label{sec:model}
Let $d \ge 2$ be the fixed dimension and  $X = \{ X_n \}_{n \geq 0}$ be a $\ZZ^d$ valued adapted process on a stochastic basis $(\Omega,\mathcal{F},\PP,\{ \mathcal{F}_n \}_{n \geq 0})$ where $\FF_0$ contains all the $\PP$-null sets of $\FF$. We denote by $\mathbb{E}$ the expectation with respect to $\PP$ and by $||\cdot||$  the euclidean norm in $\mathbb{R}^d$.
Now fix $\{\lambda_n\}_{n\ge 0}$ a sequence of positive real numbers, $\ell \in \mathbb{S}^{d-1}$, where $\mathbb{S}^{d-1}$ is the unit sphere of $\mathbb{R}^d$ and a nonempty set $A \subset \ZZ^d$. We assume that $X_0=0$ and we call $X$ a $\lambda_n$-\Name{} in direction $\ell$ with excitation set $A$, if it satisfies the following conditions:

\begin{condition}[Bounded increments]\label{condição1}
There exists a positive constant $K$ such that $\sup_{n \geq 0} || X_{n+1} - X_{n} || < K$ on every realization. 
\end{condition}

\begin{condition}\label{condição2} Almost surely
\begin{itemize}
    \item on $\{ X_k \neq X_n \, \forall \; k < n \} \cap \{ X_n \in A\}$, 
$$
    \mathbb{E} [ X_{n+1} - X_n | \mathcal{F}_n] \cdot \ell  \geq \lambda_n \, \footnote{
    Setting $B := \{ X_k \neq X_n \, \forall \; k < n \} \cap \{ X_n \in A\}$, the statement that almost surely $\mathbb{E} [ X_{n+1} - X_n | \mathcal{F}_n] \cdot \ell  \geq \lambda_n$ on the event $B$ means that 
 $\mathbbm{1}\{B\} \mathbb{E} [ X_{n+1} - X_n | \mathcal{F}_n] \cdot \ell  \geq \mathbbm{1}\{B\} \lambda_n$, almost surely. }.
$$
    \item on $\{ \exists\,  k < n  \text{ such that }  X_k = X_n \}$ or $\{ X_k \neq X_n \, \forall \; k < n \} \cap \{ X_n \notin A\}$,
    \[
     \mathbb{E} [ X_{n+1} - X_n | \mathcal{F}_n] = 0\, .
    \]
\end{itemize}
\end{condition}

\begin{condition}\label{condição3}
 There exist $h, r > 0$ such that

\begin{itemize}
    \item $X$ is {\rm uniformly elliptic in direction $\ell$}, i.e.,  for all $n$
\begin{equation}\label{eq:UE1}
\tag{UE1}\PP \left[ \left( X_{n+1} - X_n \right) \cdot \ell > r | \mathcal{F}_n \right] \geq h\,, \; {a.s..}
\end{equation}
\item $X$ is {\rm uniformly elliptic on the event $\{\mathbb{E} [ X_{n+1} - X_n | \mathcal{F}_n] = 0\}$:} on the event $ \{ \mathbb{E} [ X_{n+1} - X_n | \mathcal{F}_n] = 0 \}$,  for all $\ell' \in \mathbb{S}^{d-1}$,  with $|| \ell '|| = 1$
\begin{equation}\label{eq:UE2} 
\tag{UE2}\PP \left[ \left( X_{n+1} - X_n \right) \cdot \ell ' > r | \mathcal{F}_n \right] \geq h\,, \; {a.s..}
\end{equation}
\end{itemize}
\end{condition}

When $A = \ZZ^d$, we call $X$ simply a $\lambda_n$-\Name{}.




\medskip 
For every $\ell \in \mathbb{S}^{d-1}$, let $\mathbb{M}_{\ell}$ denote the positive half-space in direction $\ell$, that is, $\mathbb{M}_{\ell} = \{ x \in \ZZ^d : x \cdot \ell > 0 \}$.
Our main result is stated below:

\begin{theorem}\label{thm:main}
Let $X$ be a  $\lambda_n$-\Name{} in direction $\ell$ with excitation set $A \supset \mathbb{M}_{\ell}$.  There exists $\beta_0 < 1/6$ such that if for some $n_0 \in \mathbb{N}$, $\lambda >0$ and $\beta<\beta_0$, we have $\lambda_n \ge \lambda (n_0+n)^{-\beta}$ for every $n\ge 1$, then
$$
\PP \big( \lim_{n \rightarrow \infty} X_n \cdot \ell = \infty \big) > 0\,.
$$
\end{theorem}

\begin{remark}\label{rem:sub-balistic}
1. If $\lambda_n$ is $O(n^{-\beta})$ with $0 < \beta < 1/2$, then the  $\lambda_n$-\Name{} is not ballistic (thus it has null speed) since the total mean drift accumulated by time $n$ is bounded by  $n^{1-\beta}$. We conjecture that $\lim_{n \rightarrow \infty} X_n \cdot \ell = \infty$ holds almost surely, and even that $\liminf_{n \rightarrow \infty} n^{\beta - 1} (X_n \cdot \ell) > 0$ almost surely. We point out that the almost sure result is unattainable directly from the results presented here (we discuss this assertion just after the proof of Theorem \ref{thm:main} in Remark \ref{rem:justqc}.) 2. The condition $\beta < 1/6$ in the statement of Theorem \ref{thm:main} follows from limitations in our proof. We also conjecture that the result holds for $\beta < 1/2$. For a discussion and some results on the case $\beta \ge 1/2$ see \cite{AIV}.
\end{remark}



\subsection{Proof of Theorem \ref{thm:main}}\label{resultados_pn}

The strategy to prove Theorem~\ref{thm:main} is based on obtaining an analogous result to~\cite[Proposition 4.3]{menshikov2012general} where it is proved that the GERW in direction $\ell$ (with $\lambda_n$ constant) never goes below the origin in direction $\ell$ with positive probability. We generalize this proof to our case.

First, we need some auxiliary results. Given a stochastic process $\{ X_n\}_{n \geq 0}$ on the lattice $\ZZ^d$, we denote its range at time $n$ by
\begin{equation*}
\Rr_n ^X := \{ x \in \ZZ^d : X_k = x \text{ for some } 0 \leq k \leq n \}\,,
 \end{equation*}
i.e., the set of sites visited by the process up to time $n$. Henceforth $|A|$ denotes the number of elements of a set $A$. 

\begin{proposition}\label{prop41}
Let $X$ be a $\lambda_n$-\Name. Then,  there exist positive constants $\alpha \in (0, 1/6)$, $\gamma_1$, $\gamma_2$ , which depend on $K$, $h$, and $r$, such that for every  $\lambda_n$
\begin{equation*} 
\PP [|\mathcal{R}^X _n | < n^{\frac{1}{2} +\alpha} ] < \exp\{- \gamma_1 n^{\gamma_2}\}\,,  
\end{equation*}
for all $n \geq 1$. 
\end{proposition}

The proof of Proposition~\ref{prop41} is an adaptation of the proof of~\cite[Proposition 4.1]{menshikov2012general} and, for completeness, is provided in Appendix \ref{ap:A}. 
For quite general processes, a similar result is presented in ~\cite[Theorem 1.4]{menshikov2014range}. However, 
Proposition~\ref{prop41} does not follow directly  from~\cite[Theorem 1.4]{menshikov2014range}. In fact,  although the $\lambda_n$-\Name{} is a submartingale in direction $\ell$, it does not fit   in~\cite[Definition 1.1.d.]{menshikov2014range} required in~\cite[Theorem 1.4]{menshikov2014range}. 

\medskip 
Set $H(a,b) \subset \ZZ^d$ for $a<b$ as:
\begin{equation*}
H(a,b) := \{ x \in \ZZ^d : x \cdot \ell \in [a,b] \}\, , 
\end{equation*}
which represents the strip in direction $\ell$ between levels $a$ and $b$, see Figure~\ref{fig:H(a,b)}.

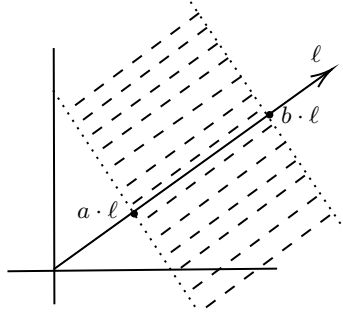
\begin{figure}[h]
    \centering

\tikzset{every picture/.style={line width=0.75pt}} 

\begin{tikzpicture}[x=0.75pt,y=0.75pt,yscale=-1,xscale=1]

\draw    (200.33,89.67) -- (200.67,219) ;
\draw    (177,202.33) -- (313,201) ;
\draw [fill={rgb, 255:red, 0; green, 0; blue, 0 }  ,fill opacity=1 ]   (200.33,201.33) -- (339.04,101.5) ;
\draw [shift={(340.67,100.33)}, rotate = 144.26] [color={rgb, 255:red, 0; green, 0; blue, 0 }  ][line width=0.75]    (10.93,-3.29) .. controls (6.95,-1.4) and (3.31,-0.3) .. (0,0) .. controls (3.31,0.3) and (6.95,1.4) .. (10.93,3.29)   ;
\draw  [dash pattern={on 0.84pt off 2.51pt}]  (200,111) -- (271.67,220.33) ;
\draw  [dash pattern={on 0.84pt off 2.51pt}]  (269.33,66.33) -- (346.33,180.33) ;
\draw  [fill={rgb, 255:red, 0; green, 0; blue, 0 }  ,fill opacity=1 ] (239.52,173.29) .. controls (239.52,172.52) and (240.14,171.9) .. (240.9,171.9) .. controls (241.67,171.9) and (242.29,172.52) .. (242.29,173.29) .. controls (242.29,174.05) and (241.67,174.67) .. (240.9,174.67) .. controls (240.14,174.67) and (239.52,174.05) .. (239.52,173.29) -- cycle ;
\draw  [fill={rgb, 255:red, 0; green, 0; blue, 0 }  ,fill opacity=1 ] (307.83,123.33) .. controls (307.83,122.57) and (308.45,121.95) .. (309.21,121.95) .. controls (309.98,121.95) and (310.6,122.57) .. (310.6,123.33) .. controls (310.6,124.1) and (309.98,124.71) .. (309.21,124.71) .. controls (308.45,124.71) and (307.83,124.1) .. (307.83,123.33) -- cycle ;
\draw  [dash pattern={on 4.5pt off 4.5pt}]  (273.67,72) -- (208.33,120.67) ;
\draw  [dash pattern={on 4.5pt off 4.5pt}]  (277.67,79.33) -- (212.33,128) ;
\draw  [dash pattern={on 4.5pt off 4.5pt}]  (283,87.33) -- (217.67,134.67) ;
\draw  [dash pattern={on 4.5pt off 4.5pt}]  (288.33,94.67) -- (223.67,141.33) ;
\draw  [dash pattern={on 4.5pt off 4.5pt}]  (295,104) -- (231,148.67) ;
\draw  [dash pattern={on 4.5pt off 4.5pt}]  (300.33,113.33) -- (237,159.33) ;
\draw  [dash pattern={on 4.5pt off 4.5pt}]  (317,137.33) -- (251.67,186) ;
\draw  [dash pattern={on 4.5pt off 4.5pt}]  (322.33,146.67) -- (255,196) ;
\draw  [dash pattern={on 4.5pt off 4.5pt}]  (327,154) -- (261.67,202.67) ;
\draw  [dash pattern={on 4.5pt off 4.5pt}]  (333,164.67) -- (267.67,213.33) ;
\draw  [dash pattern={on 4.5pt off 4.5pt}]  (339.67,174) -- (274.33,222.67) ;
\draw  [dash pattern={on 4.5pt off 4.5pt}]  (303,122) -- (239.67,168.67) ;
\draw  [dash pattern={on 4.5pt off 4.5pt}]  (310.33,128.67) -- (245,177.33) ;

\draw (328.67,87.07) node [anchor=north west][inner sep=0.75pt]  [font=\scriptsize]  {$\ell $};
\draw (210.67,165.4) node [anchor=north west][inner sep=0.75pt]  [font=\scriptsize]  {$a\cdot \ell $};
\draw (313.6,117.73) node [anchor=north west][inner sep=0.75pt]  [font=\scriptsize]  {$b\cdot \ell $};

\end{tikzpicture}

\caption{$H(a,b)$ is the strip of $\mathbb{R}^d$ corresponding to the points $x$ such that $a\le x\cdot \ell \le b$.}
\label{fig:H(a,b)}
\end{figure}

The next proposition states that if the number of sites outside the excitation set in a strip with width of order $n^{\frac{1}{2}+\alpha}$ ($\alpha$ from Proposition~\ref{prop41}) containing the origin, is also of order $n^{\frac{1}{2}+\alpha}$, then $X_n \cdot \ell\;$ is at least of order $n^{\frac{1}{2}+\alpha-\beta}$ with high probability.


\begin{proposition}\label{prop42_pnn0}
Let $X$ be a $\lambda_n$-\Name{} in direction $\ell$ with excitation set $A \subset \mathbb{Z}^d$. If for some $n_0 \in \mathbb{N}$, $\lambda >0$ and $\beta<\alpha$ (from Proposition~\ref{prop41}), we have $\lambda_n \ge \lambda (n_0+n)^{-\beta}$ for every $n\ge 1$, and if for some $n \geq n_0$ it holds that 
\begin{equation}\label{cond44_pnn0}
\left\vert( \mathbb{Z}^d \setminus A) \cap H\left( - n^{\frac{1}{2} + \alpha}, \frac{2 \lambda}{3} n^{\frac{1}{2} + \alpha}\right)\right\vert \leq \frac{1}{3} n^{\frac{1}{2} + \alpha}\,,
\end{equation}
then 
\begin{equation}\label{eq42c_pnn0}
\PP\left[ X_n \cdot \ell < \frac{1}{3} \lambda n^{\frac{1}{2}+ \alpha-\beta} \right] < 5n\exp\{-\vartheta_1 n^{\vartheta_2}\}\, , 
\end{equation}
where 
\begin{align*}
\vartheta_1 & = \min\left\{ \gamma_1, \frac{1}{2K^2}, \frac{\lambda^{2}}{18K^2}, \frac{((1/3 -2^{1-\beta}/3) \lambda)^2}{2K^2} \right\} \,,
\\
\vartheta_2 & = \min\left\{ \gamma_2, 2(\alpha-\beta) \right\} \,,
\end{align*}
and $\gamma_1$, $\gamma_2$ are the same as in Proposition~\ref{prop41}.
\end{proposition}

Let us point out that the constant $\lambda$ appearing in the statement of Proposition~\ref{prop42_pnn0}, which controls the strength of the drift,  is the same constant appearing  in~\eqref{cond44_pnn0} and \eqref{eq42c_pnn0}. The fact that the width of the strip $H$ depends on $\lambda$ is a technical choice used in the proof of Proposition~\ref{prop42_pnn0}.
%
The proof is provided in Appendix \ref{ap:B} and is an adaptation of the proof of~\cite[Proposition 4.2]{menshikov2012general}


\medskip 

The next proposition is our main contribution. It states that the $\lambda_n$-GERW in direction $\ell$ never goes below the origin in that direction with positive probability bounded below by some $\psi$ whose dependence on $\beta$ and $n_0$ is explicitly described. 



\smallskip 
\begin{proposition}\label{prop43_pnn0-1} 
Let $X$ be a  $\lambda_n$-\Name{} in direction $\ell$ with excitation set $A \supset \mathbb{M}_{\ell}$. If for some $n_0 \in \mathbb{N}$, $\lambda >0$ and $\beta < \alpha$ (same $\alpha$ from Proposition~\ref{prop41}) we have $\lambda_n \ge \lambda (n_0+n)^{-\beta}$ for every $n\ge 1$, then there exists a  constant $\psi>0$ such that
\begin{equation*}
\PP\left[ X_n \cdot \ell > 0 \text{  for all  } n\geq 1\right] \geq \psi\, .
\end{equation*}
\end{proposition}

\begin{remark}
It is worth mentioning that we know explicitly how the constant $\psi$ appearing in Proposition~\ref{prop43_pnn0-1} depends on the parameters of the model. Specifically,   
\[
\psi = h^{\lceil r^{-1} \rceil C \left(\frac{3}{\lambda} \right)^{\frac{1}{\delta -1}}} c\,,
\]  with $c \in (0, 1)$, $\delta = (2-\alpha+\beta)(1/2 + \alpha-\beta)$, 
$C \ge K^{\frac{1}{\delta-1}} \left( \eta + \lceil r^{-1} \rceil^{\frac{1}{\delta-1}}\right) + n_0$, where 
\begin{align*}
\eta & = \left( \frac{ 2-\alpha+\beta}{\vartheta_1 \varphi_1}\right)^{\frac{1}{\varphi_1}}  \, , \quad \varphi_1  = \min \left\{ \alpha-\beta, (2-\alpha+\beta)\vartheta_2 \right\}\, , \nonumber
\end{align*}
and $\vartheta_1$, $\vartheta_2$ are as in Proposition~\ref{prop42_pnn0}. 
\end{remark}
Now we use Proposition~\ref{prop43_pnn0-1} to prove Theorem~\ref{thm:main}, and we prove the proposition just after. 

\smallskip

\begin{proof}[Proof of Theorem~\ref{thm:main}]
Note that 
\begin{equation*}
\PP\big[ \lim_{n \to \infty} X_n \cdot \ell = \infty \big] \geq  \PP[X_n \cdot \ell > 0 \text{ } \forall n \ge 1] \PP\big[ \lim_{n \to \infty} X_n \cdot \ell = \infty \mid X_n \cdot \ell > 0 \text{ } \forall n \ge 1 \big] \,,
\end{equation*}
and by Proposition~\ref{prop43_pnn0-1} we know that $ \PP[X_n \cdot \ell > 0 \text{ } \forall n \ge 1]\geq\psi>0$. Therefore, it suffices to prove that   $\PP\big[ \lim_{n \to \infty} X_n \cdot \ell = \infty \mid X_n \cdot \ell > 0 \text{ } \forall n \ge 1 \big]=1$. 

Fix $m \in \mathbb{N}$ and set $\mathcal{F}_n$-stopping times $\nu_{m,0} \equiv 0$,
$$
\tau_{m,j} = \inf \{ n > \nu_{m,j-1} : X_n \notin H(0,m) \}\,, 
$$
and
$$
\nu_{m,j} = \inf \{ n > \tau_{m,j} : X_n \in H(0,m) \}\,, 
$$
for every $j\ge 1$. Then,  $\tau_{m,j}-\nu_{m,j-1}$ represents the duration of the $j$-th excursion outside the strip $H(0,m)$. Note that, as usual, if $\tau_{m,j} = \infty$  (resp. $\nu_{m,j-1} = \infty$) then $\nu_{m,j} = \infty$ (resp. $\tau_{m,j} = \infty$). Also define $\mathcal{G}_j = \mathcal{F}_{\nu_{m,j}}$ and 
$$
\Gamma_j =  \{\nu_{m,j}<\infty\} \cap \big\{ X_n \cdot \ell < 0 \textrm{ for some } n \in (\nu_{m,j}, \nu_{m,j} + \hat m] \big\}  \,,
$$
where $\hat m = \lfloor m/r \rfloor + 1$. That is, $\Gamma_j$ denotes the event that the $j$-th excursion is finite, i.e.,  the random walk returns to the strip in a finite time after leaving it for the $j$-th time,  and that in the subsequent $\hat m$ steps after returning it visits the region $\{x\cdot \ell<0\}$. Recall that our random walk does not have independent increments nor it has the strong Markov property, therefore $\PP (\Gamma_{ j}\mid\mathcal{G}_{j})$ depends on the whole history until time $\nu_{m,j}$. The sets $\Gamma_{\hat m j}$ are $\mathcal{G}_{\hat m (j+1)}$-measurable and Condition~\ref{condição3} implies that
$$
\PP (\Gamma_{\hat m j}\mid\mathcal{G}_{\hat m j}) \geq \mathbbm{1}\{\nu_{m,\hat{m}j}<\infty\}h^{\hat m}\,.
$$
One can see that 
\begin{equation*}
\begin{split}
& \PP\Big[ \sum_{j=1}^{\infty} \PP (\Gamma_{\hat m j}\mid\mathcal{G}_{\hat m j}) = \infty \Bigm| \nu_{m,i} < \infty \, \forall \, i\ge 1  \Big] 
\ge \PP\Big[ \sum_{j=1}^{\infty} h^{\hat m} = \infty \Bigm| \nu_{m,i} < \infty \, \forall \, i\ge 1  \Big] = 1 \,.
\end{split}  
\end{equation*}
Hence, by a generalization of the second Borel-Cantelli lemma \cite[Theorem 5.3.2]{durrett}, since $\{\mathcal{G}_{\hat{m}j}\}_{j \ge 0}$ is a filtration and $\Gamma_{\hat{m}j} \in \mathcal{G}_{\hat{m}(j+1)}$ for all $j \ge 1$, we have that, a.s., 
\begin{equation}\label{eq:Borel}
    \{ \Gamma_{\hat{m}j} \ i.o. \} = \Big\{ \sum_{j=1}^{\infty} \PP (\Gamma_{\hat m j}\mid\mathcal{G}_{\hat m j}) = \infty \Big\} \,.
\end{equation} 

Ergo, by~\eqref{eq:Borel}, we obtain that 
$\PP \big( \Gamma_{\hat m j} \ i.o. \bigm| \nu_{m,j} < \infty \, \forall \, j\ge 1 \big) = 1$, and therefore
$$
\PP \big( \Gamma_{j} \ i.o. \bigm| \nu_{m,j} < \infty \, \forall \, j\ge 1 \big) = 1\,.
$$ 
The latter equation reveals that if the random walk keeps on coming back to the strip, it must visit the region $\{x\cdot \ell<0\}$ infinitely often almost surely. Thus, on $\{X_n \cdot \ell > 0 \ \forall n\geq 1 \}$ we have that the event $\{\textrm{there exists } j=j(m)\ge 1 \textrm{ such that } \nu_{m,j} = \infty \}$ must occur almost surely. 

Therefore, on the event $\{X_n \cdot \ell > 0 \ \forall n\geq 1 \}$ we indeed have that 
$$
\bigcap_{m \ge 1} \{\textrm{there exists } j=j(m)\ge 1 \textrm{ such that } \nu_{m,j} = \infty \}\,,
$$ 
occurs almost surely, which implies that $\{ \lim_{n \rightarrow \infty} X_n \cdot \ell = \infty \}$ almost surely. 
%
\end{proof}

\smallskip

\begin{remark}\label{rem:justqc}
Set $\zeta_K = \inf\{ m : X_m \cdot \ell \ge K \}$, $K>0$. Following the proof of  Theorem~\ref{thm:main} we obtain that $\Xi = \lim_{n \rightarrow \infty} X_n \cdot \ell = \infty$ in $\{ X_n \cdot \ell \ge K \ \forall n \ge \zeta_K 
\textrm{ for some } K\}$. Thus we would have the almost sure result in Theorem~\ref{thm:main} if $\Xi$ has probability one. We conjecture that this is true, but here the drift is reduced over time and we need a fine control on times $\zeta_K$ to be related to the lower bound in Proposition~\ref{prop43_pnn0-1}.
\end{remark}

\smallskip

\begin{proof}[Proof of Proposition~\ref{prop43_pnn0-1}]
Since on $\{ X_n \cdot \ell > 0 \text{ for all } n \geq 1 \}$ the process does not visit $\mathbb{Z}^d \backslash \mathbb{M}_{\ell}$,  it is sufficient to consider the case $A= \mathbb{Z}^d$.

Without loss of generality we consider $r\le 1$ in Condition~\ref{condição3} and $\lambda<1$. Define
\begin{equation*}
U_0 = \left\{ \left( X_{k+1} - X_k \right)\cdot \ell \geq r\;,  \text{ for all} \ k = 0, 1, ..., \lceil r^{-1} \rceil m - 1 \right\}.  
\end{equation*}
Note that $X_{\lceil r^{-1} \rceil m} \cdot \ell \geq m$ in $U_0$ and by \eqref{eq:UE1} in Condition~\ref{condição3}
\begin{equation}\label{u0_pn}
\PP \left[ U_0 \right] \geq h^{\lceil r^{-1} \rceil m}\,.
\end{equation}

Consider the following time shift of $X$: $W_k = X_{\lceil r^{-1} \rceil m + k}$, $k \geq 0$. Then $W$ is a $\tilde{\lambda}_{n}$-\Name{} (with $\tilde{\lambda}_n=\lambda_{n+\lceil r^{-1} \rceil m}$) adapted to the filtration $\{\FF_{k + \lceil r^{-1} \rceil m}\}_{k \ge 0}$ with excitation set 
$$
A' = \mathbb{Z}^d \setminus \{ X_0, \dots, X_ {\lceil r^{-1} \rceil m-1} \}\,,
$$ 
starting at $W_0 = y_0 := X_{\lceil r^{-1} \rceil m}$. Moreover for every $k$ we have 
\[
\EE[W_{k+1} - W_k | \FF_{k +\lceil r^{-1} \rceil m }] \cdot \ell  \ge \lambda(n_0 + \lceil r^{-1} \rceil m + k)^{-\beta}\,,
\] 
almost surely  on 
$\{ W_j \neq W_k \text{ for all } \; j < k \} \cap \{ W_k \in A'\}$.
%
%
Set $\delta = (2-\theta)(1/2 + \theta)$, where $\theta=\alpha-\beta$ and
$$m = C \Big(\frac{3}{\lambda }\Big)^{\frac{1}{\delta -1}}\,,$$
with $C$ as in the
statement of  Proposition~\ref{prop43_pnn0-1}.
Since $0<\beta <\alpha<1/6$, we have that $\delta > 1$.

The left-hand side of \eqref{cond44_pnn0} with the set $A$ taken as $A' - y_0$ is bounded above by $\lceil r^{-1} \rceil m$. Note that, for all $n \geq  m^{2- (\alpha-\beta)}$,
\begin{align*}
\frac{1}{3}n^{\frac{1}{2} + \alpha} \geq \frac{1}{3}n^{\frac{1}{2} + \alpha-\beta}  \geq \frac{1}{3} m^{\delta}  \geq \lceil r^{-1}\rceil m\,,
\end{align*}
where the last inequality follows from 
\begin{align}\label{nm_pn2}
\frac{m^{\delta-1}}{3\lceil r^{-1} \rceil} = \frac{3C^{\delta-1}}{3\lceil r^{-1} \rceil \lambda} > \frac{K\lceil r^{-1} \rceil}{\lceil r^{-1} \rceil \lambda}  = \frac{K}{\lambda} > 1 \, ,
\end{align}
since $K \ge 1 > \lambda$. The second inequality above follows from the definition of $C$, since $C > (K \lceil r^{-1}\rceil)^{\frac{1}{1-\delta}}$ and  $K\geq 1$. Thus~\eqref{cond44_pnn0} with excitation-allowing set $A' - y_0$ is satisfied for all $n\geq m^{2- \alpha+\beta}$.

Denote $m_0 = 0$, $m_1 = m$ and, for  $k \geq 1$,  $m_{k+1} = \frac{1}{3} \lambda m_{k}^{\delta}$.   The sequence $\{m_k\}_{k \geq 1}$ is increasing.  This can be proved by induction since  
\[ \frac{m_2}{m_1} = \frac{\lambda}{3} m^{\delta - 1} = C^{\delta-1}  > 1\,,\]
for all $\theta \in (0, \alpha)$ (which implies $\delta>1$), and assuming $m_k/m_{k-1} > 1$ we have
\[ \frac{m_{k+1}}{m_k} = \frac{ \frac{\lambda}{3} m_k ^{\delta}}{\frac{\lambda}{3} m_{k-1} ^{\delta}} = \left( \frac{m_k}{m_{k-1}} \right)^{\delta} > 1\,.  \]

For every $k \geq 1$ consider the following events:
$$
G_k = \Big\{ \min_{\lfloor m_{k-1} ^{2 - \theta}\rfloor  < j \leq m_k ^{2 - \theta}} \big( W_j - W_{\lfloor m_{k-1}^{2- \theta}\rfloor} \big) \cdot \ell > - m_k \Big\} 
$$
and
$$
U_k = \big\{ W_{\lfloor m_{k} ^{2 - \theta}\rfloor } \cdot \ell \geq m_{k+1} \big\}.    
$$
One can see that 
\begin{equation}\label{inclusão_pnn0}
\big\{ X_n \cdot \ell > 0, \text{ for all } n \geq 1 \big\} \supset \Big( \bigcap_{k=1}^{\infty} \left(G_k \cap U_k \right) \Big) \cap U_0\, .    
\end{equation}
The process $\{X_n \cdot \ell\}_{n\geq 0}$, is a $\FF$-submartingale, so $(W - y_0) \cdot \ell$ is also $\FF_{\cdot +\lceil r^{-1} \rceil m }$-submartingale. Write
\begin{equation*}
G_k^c = \bigcup_{j= \lfloor m_{k-1}^{2-\theta} \rfloor + 1}^{m_k^{2-\theta}} \left\{ \left( W_j - W_{\lfloor m_{k-1}^{2-\theta} \rfloor}\right) \cdot \ell \leq -m_k \right\}\,.    
\end{equation*}
and by Azuma's inequality (for supermartingales with bounded increments) 
\begin{align*}
\PP&\left[ \left( W_j - W_{\lfloor m_{k-1}^{2-\theta} \rfloor}\right) \cdot \ell \leq -m_k \right]  = \PP\left[ \left( W_{\lfloor m_{k-1}^{2-\theta} \rfloor} -W_j\right) \cdot \ell \geq m_k \right] 
\\
& \leq \exp \Big( - \frac{m_k^2}{2K^2(j - \lfloor m_{k-1}^{2-\theta} \rfloor) } \Big)
 \leq \exp \Big( - \frac{m_k^2}{2K^2 m_{k}^{2-\theta}} \Big)
 \leq \exp \Big( - \frac{m_k^{\theta}}{2K^2} \Big)\,.
\end{align*}
Thus, we have
\begin{equation*}
\PP[G_k|U_0] \geq 1 -\left( m_k ^{2- \theta} -\lfloor m_{k-1}^{2- \theta}\rfloor\right) e^{- \frac{m_k ^{\theta}}{2K^2}} \geq 1 - m_k ^{2- \theta} e^{- \frac{m_k ^{\theta}}{2K^2}}\,.
\end{equation*}


We point out that $m^{2-\theta}\geq n_0 + \lceil r^{-1} \rceil m$, indeed 
\begin{align*}
m^{1-\theta} - \lceil r^{-1} \rceil \geq \left( \frac{3K\lceil r^{-1} \rceil}{\lambda} \right)^{\frac{1-\theta}{\delta-1}} -\lceil r^{-1} \rceil > 1 \,, 
\end{align*}
since for all $\theta \in (0, \alpha)$ we have $(1-\theta)/(\delta - 1)> 1$. Thus
\begin{align*}
m^{2-\theta} - \lceil r^{-1} \rceil m & =  m \underbrace{( m^{1-\theta} - \lceil r^{-1} \rceil )}_{ > 1} \geq m \geq n_0
 \,.
\end{align*}
Since $\{m_k\}_{k \geq 1}$ is increasing, then $m_k^{2-\theta}\geq n_0 + \lceil r^{-1} \rceil m$ for all $k \geq 1$. Now the process $W-y_0$ satisfies Conditions~\ref{condição1},~\ref{condição2},~\ref{condição3}, the set $A'-y_0$ fulfills~\eqref{cond44_pnn0} for all $n \geq m^{2-\theta}$ and $m_k^{2-\theta}\geq n_0 + \lceil r^{-1} \rceil  m$ for all $k \ge 1$.  By Proposition~\ref{prop42_pnn0}, it holds that
\begin{equation*}
\PP[U_k|U_0] = \PP\left[W_{\lfloor m_k ^{2-\theta} \rfloor} \cdot \ell \geq \frac{\lambda }{3} m_{k}^{(2-\theta)(\frac{1}{2}+\theta)}
\right] \geq 1 - 5m_k^{2-\theta} e^{- \vartheta_1 m_k ^{(2- \theta)\vartheta_2}}\,.  
\end{equation*}

Now, write
\[ 
\PP \left[ \left( \bigcap_{k=1}^{\infty} \left(G_k \cap U_k \right) \right) \cap U_0 \right] = \PP[U_0] \Big( 1 - \sum_{k=1}^{\infty} ( \PP[G_k ^c| U_0] + \PP[U_k ^c| U_0] ) \Big)\,,
\] 
which is bounded from below by
\begin{equation}\label{psi_pnn0}
\begin{split}
 &\PP[U_0] \!\left( \!1 -\! \sum_{k=1}^{\infty}\!\left( m_k ^{2- \theta} e^{- \frac{m_k ^{\theta}}{2K^2}} + 5m_k ^{2- \theta}  e^{- \vartheta_1 m_k ^{(2- \theta)\vartheta_2}} \right)\!\right)
\!\geq\! h^{\lceil r^{-1} \rceil m} \!\left(\! 1 \!- 6\sum_{k=1}^{\infty} m_k ^{2- \theta} e^{- \vartheta_1 m_k ^{\varphi_1}}\! \right)\,.
\end{split}
\end{equation}
Note that $m$ is large enough so that $\{m_{k}^{2-\theta}e^{- \vartheta_1 m_k^{\varphi_1}}\}_{k \geq 1}$ is decreasing. Indeed,  $m$ is bigger than the inflection point $\left( \frac{ 2-\theta}{\vartheta_1 \varphi_1}\right)^{\frac{1}{\varphi_1}}$ of the function $z(x)=x^{2-\theta}e^{-\vartheta_1 x^{\varphi_1}}$, $x>0$:
\begin{align*}
m = C\left( \frac{3}{\lambda }\right)^{\frac{1}{\delta-1}} & > K^{\frac{1}{\delta-1}} \eta \left( \frac{3}{\lambda}\right)^{\frac{1}{\delta-1}} 
 = \eta \left(\frac{3K}{\lambda} \right)^{\frac{1}{\delta-1}} \geq \left( \frac{ 2-\theta}{\vartheta_1 \varphi_1}\right)^{\frac{1}{\varphi_1}}\,. 
\end{align*}
Thus, since we have that $m_{k+1}-m_k \geq  m \left( C^{\delta-1}-1 \right)>1$, it holds that 
\begin{align*}
\sum_{k=1}^{\infty}& m_{k}^{2-\theta}e^{-\vartheta_1 m_k^{\varphi_1}}  \leq 
m^{2-\theta}e^{-\vartheta_1 m^{\varphi_1}}
 + 
\int_{m}^{\infty} x^{2-\theta} e^{-\vartheta_1 x^{\varphi_1}} dx \, .
\end{align*}
By a change of variables, we write,
\begin{align*}
&\int_{m}^{\infty} x^{2-\theta_1} e^{-\vartheta_1 x^{\varphi_1}} dx  =  \varphi_1^{-1} \vartheta_1^{\frac{\theta-3}{\varphi_1}}\Gamma \left( \frac{3-\theta}{\varphi_1}, \vartheta_1 m^{\varphi_1} \right)\, ,
\end{align*}
where $\Gamma$ is the incomplete gamma function\footnote{$\Gamma(s, x) = \int_{x}^{\infty} t^{s-1}e^{-t} dt$.}. 
As mentioned above $m$ is large enough so that the sequence $\{m_{k}^{2-\theta}e^{- \vartheta_1 m_k^{\varphi_1}}\}_{k \geq 1}$ is decreasing. 
If needed, 
we may increase $m$ even further by increasing $C$, in order to  obtain
\begin{align}\label{eq: 1/7_pnn0}
\sum_{k=1}^{\infty} m_{k}^{2-\theta}e^{-\vartheta_1 m_k^{\varphi_1}} & \leq  
m^{2-\theta}e^{-\vartheta_1 m^{\varphi_1}} +   \varphi_1^{-1} \vartheta_1^{\frac{\theta-3}{\varphi_1}}\Gamma \left( \frac{3-\theta}{\varphi_1}, \vartheta_1 m^{\varphi_1} \right) \le \frac{1}{7}\, . 
\end{align}

Using~\eqref{eq: 1/7_pnn0} in~\eqref{psi_pnn0}, we obtain that,
\begin{align*}
\PP  \left[ \left( \bigcap_{k=1}^{\infty} \left(G_k \cap U_k \right) \right) \cap U_0 \right]
 &\geq h^{\lceil r^{-1} \rceil m} \left( 1 - 6\sum_{k=1}^{\infty} m_k ^{2- \theta} e^{- \vartheta_1 m_k ^{\varphi_1}} \right)
\geq  h^{\lceil r^{-1} \rceil C\left(\frac{3}{\lambda } \right)^{\frac{1}{\delta -1}}} c = \psi \, ,
\end{align*}
where $c \in (0,1)$. Theorem~\ref{thm:main} then follows from $\eqref{inclusão_pnn0}$.
\end{proof}

\medskip
\paragraph{\bf Acknowledgements:} We would like to thank Augusto Quadros Teixeira, Christophe Gallesco, Guilherme Ost, Luiz Renato Fontes and Maria Eulalia Vares for useful comments and suggestions.

\medskip

\paragraph{\bf Conflict of interest:} The authors have no competing interests to declare that are relevant to the content of  this article.

\medskip

\paragraph{\bf Data Availability:} Data sharing is not applicable to this article as no datasets were generated or analyzed during
the current study.


\appendix
\section{Proof of Proposition~\ref{prop41}}\label{ap:A}

This section is devoted to the proof of  Proposition~\ref{prop41}. Let $X$ be a $\lambda_n$-\Name{}. We denote by $L_n (m)$ its local time up to time $n$ on the $m$-th strip in direction $\ell$ which is defined as 
\begin{align*}
 L_n (m)  := \sum_{j=0}^n 1_{ \{ X_j \cdot \,\ell \; \in \,[m, m+1) \}}\;, \; m \in \ZZ \, .
\end{align*}
Let $T_{U}^X$ denote the first time the process $X$ visits the set $U$, i.e.,  
\begin{equation*}
    T_U^X := \min \{ n \ge 0 : X_n \in U \}\,.
\end{equation*}

The next result provides a control on the tail of the local times $L_n (m)$. 

\begin{lemma}\label{lema51}
If $X$ satisfies Conditions~\ref{condição1}, ~\ref{condição2} and~\ref{condição3}, then for any $\delta > 0$ there exists $\gamma_1' = \gamma_1'(K,h,r)$ such that 
\begin{equation*}
\PP \left[ L_n (m) \geq n^{\frac{1}{2} + 2\delta} \right] \leq \exp\{- \gamma_1' n^{\delta}\}\, , \ \forall \, m \in \ZZ\,.  
\end{equation*}
\end{lemma}

The proof of Lemma~\ref{lema51} is the same as that of~\cite[Lemma 5.1]{menshikov2012general} which relies on the uniform elliptic condition and the fact that $\{X_n \cdot \ell\}_{n \ge 0}$ is a $\FF$-submartingale and both properties are valid in our case.  
Using Lemma~\ref{lema51} and~\cite[Lemma 5.4]{menshikov2012general} we now prove Proposition~\ref{prop41}.

\begin{proof}[Proof of Proposition \ref{prop41}]
Consider $b \in (0,1)$ and $\varepsilon > 0$. Recall that when the process is in an already visited site it behaves like a $d$-martingale with bounded jumps and uniform elliptic condition; we denote this process by $\{Y_n\}_{n \ge 0}$. Then by~\cite[Lemma 5.2]{menshikov2012general} there exists a $b' \in (0, 1)$ and a positive constant $c$ (depending of $K$, $h$ and $r$) such that
\[ \EE[||Y_{n+1}||^{b'} | \FF_n] \ge ||Y_n||^{b'} 1_{\{||Y_n|| \ge c\}}\,.\]

Set $b = b'$ and $0< e_w < 1/6$ and define
\begin{equation*}
H^n_j := H\big( 2(j - 1)n^{e_w}, 2(j + 1)n^{e_w} \big)\, , \ n\ge 1 \,, \ j\ge 1\,,
\end{equation*}
so that $H^n_j$ is strip of width $4n^{e_w}$ in direction $\ell$. The strip $H^n_j$ will be called a trap if $|\Rr_n ^X \cap H^n_j | \geq n^{e_t}$, where $e_t = 2e_w(1-(b/2)) - 2 \varepsilon$ is the trap exponent. 
%
Set
\begin{equation*}
G = \{ |\Rr_n ^X | \geq n^{\frac{1}{2} + e_w (1-b)-4\varepsilon} \}\, .   
\end{equation*}
We are going to prove that
\begin{equation}\label{probg}
\begin{split}
\PP[G] \geq 1 - \left( \left( 2Kn + 1 \right)e^{-\gamma_1' n^{\frac{\varepsilon}{2}}} + \frac{n^{1-2e_w + \varepsilon}}{2} e^{-\gamma_4' n^{\varepsilon}} \right)\,,
\end{split}
\end{equation} 
for every $\varepsilon >0$ sufficiently small. 
This would establishes Proposition $\ref{prop41}$ since, for $\alpha < e_w(1-b)-4\varepsilon$ we have that  $\{|\Rr_n ^X | < n^{\frac{1}{2} + \alpha} \} \subset G^c$. 
%
%
Towards proving \eqref{probg}, let us introduce the event
\begin{equation*}
G_1 = \big\{ L_n (k) \leq n^{\frac{1}{2} + \varepsilon} \text{ for all } k \in [-Kn, Kn] \big\}\,. 
\end{equation*}
By Lemma~\ref{lema51}, it holds that
\begin{equation}\label{probg1}
\PP[G_1] \geq 1 - (2Kn + 1)\exp\{-\gamma_1' n^{\frac{\varepsilon}{2}}\}\, .    
\end{equation}
Now, let us define $\sigma_0 = 0$ and inductively 
\begin{equation}\label{sigmak}
\sigma_{k+1} = \min \{ j \geq \sigma_k + \lfloor n^{2e_w - \varepsilon} \rfloor : | \Rr_j ^X \cap B(X_j, n^{e_w}) | \leq n^{e_t} \}\,,  
\end{equation} 
(formally, if such $j$ does not exist, we put $\sigma_{k+1} = \infty$). Consider the event 
\begin{align*}
G_2 = \Big\{ & \text{at least one new point is hit on each of the time intervals } 
\\
& [\sigma_{j-1}, \sigma_j), j = 1,..., \frac{1}{2}n^{1-2 e_w + \varepsilon} \Big\}\,,  
\end{align*}
where hitting a new point means visiting a not-yet-visited site. Note that on $G^c_2$, the process does not hit a new point in time interval $[\sigma_{j-1},\sigma_j)$ for some $j=1,..., \frac{1}{2}n^{1-2 e_w +\varepsilon}$. When this happens, the process $X$ evolves as a $d$-dimensional martingale during time interval $[\sigma_{j-1},\sigma_j)$ and for $Y_\cdot = X_{\sigma_{j-1}+\cdot}$ 
$$
T^Y_{\left(\mathcal{R}^X_{\sigma_{j-1}}\right)^c} \ge \sigma_j - \sigma_{j-1} \ge n^{2e_w-\varepsilon}\,.
$$



To control the probability of $G_2^c$, we will apply~\cite[Lemma 5.4]{menshikov2012general}, setting $\delta = \frac{\varepsilon}{2 e_w}$, $m = n^{2 e_w}$ and $U = \left(\mathcal{R}^X_{\sigma_{j-1}}\right)^c$.
By the definition of $\sigma_{j-1}$, we have that $|\Rr_{\sigma_{j-1}} ^X \cap B(X_{\sigma_{j-1}}, n^{e_w}) | \leq n^{e_t}$, which for our choice of parameters implies 
\begin{align*}
| B(X_{\sigma_{j-1}}, n^{e_w}) \setminus \left(\Rr_{\sigma_{j-1}} ^X \right)^c| &=| B(X_{\sigma_{j-1}}, (n^{2e_w})^{1/2}) \setminus \left(\Rr_{\sigma_{j-1}} ^X \right)^c|  \leq n^{e_t} 
\\
&
= n^{2e_w(1- b/2) -2\varepsilon}=\left(n^{2e_w}\right)^{1-b/2 - 2\delta}\,,  
\end{align*}
and thus we can use~\cite[Lemma 5.4]{menshikov2012general} to conclude that 
$$
\PP\left[T^Y_{\left(\mathcal{R}^X_{\sigma_{j-1}}\right)^c} \ge \sigma_j - \sigma_{j-1} \ge \left(n^{2 e_w} \right)^{1 - \frac{\varepsilon}{2e_w}}\right] \le e^{-\gamma_4' n^{\delta}}\,,
$$
where $\gamma_4^\prime$ is a positive constant. Thus
\begin{align}\label{probg2}
\PP[G_2]  
 \geq 1 - \frac{1}{2}n^{1-2 e_w+\varepsilon} e^{-\gamma_4' n^{\delta}}\,.
\end{align}

Next, assuming that $n$ is large enough so that $8n^{1-\varepsilon} < n/2$,  we will show that $(G_1 \cap G_2) \subset G$. Suppose that both $G_1$ and $G_2$ occur, but $|\Rr_n ^X | < n^{\frac{1}{2} + e_w (1-b)-4\varepsilon}\}$. Denote by $\hat{L}_j $ the number of visits to $H^n_j$ up to time $n$, i.e., 
\begin{equation*}
\hat{L}_j = \sum_{k = 2(j-1)n^{e_w}}^{2(j+1)n^{e_w} - 1} L_n (k)\,,    
\end{equation*}
On $\{|\Rr_n ^X | < n^{\frac{1}{2} + e_w (1-b)-4\varepsilon}\}$ we have
$$
n^{\frac{1}{2} + e_w (1-b)-4\varepsilon} > |\Rr_n ^X| \ge  n^{e_t} |\{j: H^n_j  \text{ is a trap}\}|\,,
$$
thus the number of traps is at most 
$$
2n^{\frac{1}{2} + e_w(1-b) - 4\varepsilon -e_t} = 2n^{\frac{1}{2} - e_w - 2\varepsilon}\,.
$$
On $G_1$, we have, 
\begin{equation*}
\sum_{j\in \ZZ} \hat{L}_j 1_{\{H^n_j  \text{ is a trap}\}} \leq 4n^{e_w} \times 2n^{\frac{1}{2} - e_w - 2\varepsilon} \times n^{\frac{1}{2} + \varepsilon} = 8n^{1-\varepsilon}\,.  
\end{equation*}
Now observe that, since for $j \leq n$ we have $\Rr_j ^X \subset \Rr_n ^X$, if $|\Rr_j ^X \cap B(X_j, n^{e_w})| > n^{e_t}$ then $X_j$ must be in a trap.
Since $n$ is such that $8n^{1-\varepsilon} < n/2$, we obtain that,  on the event
\begin{equation*}
\Big\{ \sum_{j \in \ZZ} \hat{L}_j 1_{\{H^n_j  \text{ is a trap}\}} \leq 8n^{1-\varepsilon} \Big\}\,, 
\end{equation*}
the total time (up to time $n$) spent in non-traps is at least $n-8n^{1-\varepsilon} > n/2$. From~\eqref{sigmak}, the latter implies that   $\sigma_{\frac{n^{1-2 e_w + \varepsilon}}{2}} < n$. Indeed,  up to time $\sigma_{\frac{n^{1-2 e_w + \varepsilon}}{2}}$ we can have at most $n/2$ instances $j$ such that $|\Rr_j ^X \cap B(X_j, n^{e_w})| \leq n^{e_t}$. Therefore, on the event $G_2$ we have that $|\Rr_n ^X| \geq \frac{1}{2}n^{1-2e_w + \varepsilon}$. Recall that we assumed that  $G_1$ and $G_2$ occur, but $|\Rr_n ^X | < n^{\frac{1}{2} + e_w (1-b)-4\varepsilon}$. Since for  $e_w <\frac{1}{6}$ (and $n$ sufficiently large) it holds that $\frac{1}{2}n^{1-2e_w + \varepsilon} > n^{\frac{1}{2} + e_w (1-b)-4\varepsilon}$, for every $b \in (0,1)$ and $\varepsilon>0$, we obtain a contradiction. Then, $(G_1 \cap G_2) \subset G$, and~\eqref{probg} follows from~\eqref{probg1} and          ~\eqref{probg2},
\begin{align*}
\PP[G] & \geq \PP[G_1 \cap G_2]
 \geq 1 - (\PP[G_1] + \PP[G_2])
\geq 1 - \left( \left( 2Kn + 1 \right)e^{-\gamma_1' n^{\frac{\varepsilon}{2}}} + \frac{n^{1-2e_w +\varepsilon }}{2} e^{-\gamma_4' n^{ \varepsilon}} \right)\,.
\end{align*}
To conclude  the proof of Proposition $\ref{prop41}$, just note that for every $\alpha<1/6$, we can find $e_w<1/6$ and $b \in (0,1)$ and $\varepsilon$ (sufficiently small), such that $\alpha< e_w(1-b)-4\varepsilon$. \end{proof}

\section{Proof of Proposition~\ref{prop42_pnn0}}\label{ap:B}

Let $G$ be the following event
\begin{equation*}
G:=\left\{ |\mathcal{R}^X _n | \geq n^{\frac{1}{2} + \alpha} \right\} \cap \left\{X_k \in H\left( -n^{\frac{1}{2} + \alpha}, \frac{2}{3} \lambda n^{\frac{1}{2} + \alpha} \right), \text{ for all } k \leq n \right\} \,.
\end{equation*}

In order to prove \eqref{eq42c_pnn0}, we write that probability as   
\begin{align}\label{prop42p1_pn}
& \PP\left[ \left\{ X_n \cdot \ell < \frac{1}{3} \lambda n^{\frac{1}{2}+ \alpha-\beta} \right\} \cap G \right] + 
\PP\left[ \left\{ X_n \cdot \ell < \frac{1}{3} \lambda n^{\frac{1}{2}+ \alpha - \beta} \right\} \cap G^c \right] \,,
\end{align}
and we control the two terms separately. Note that the process $\{X_n\cdot \ell\}_{n\geq 0}$ is a $\FF$-submartingale and thus $\{-X_n\cdot \ell\}_{n \geq 0}$ is a $\FF$-supermartingale. 
As regards the second term in \eqref{prop42p1_pn},  set
$$E = \left\{ X_n \cdot \ell < \frac{
1}{3} \lambda n^{\frac{1}{2}+ \alpha - \beta}\right\}, \  M = \left\{ |\mathcal{R}^X _n | < n^{\frac{1}{2} + \alpha} \right\},$$ 
$$J = \left\{ \min_{k \leq n} X_k \cdot \ell < -n^{\frac{1}{2} + \alpha} \right\} \textrm{ and } T = \left\{ \max_{k \leq n} X_k \cdot \ell >   \frac{2}{3} \lambda n^{\frac{1}{2} + \alpha} \right\}.$$ 
It follows that
\begin{equation*}
\begin{split}
\PP&\left[ E \cap G^c \right]  = \PP\left[ \left( E \cap M \right) \cup \left( E \cap J \right) \cup \left( E \cap T\right)   \right]
\\
& \leq \PP\left[E \cap M \right] + \PP\left[ E \cap J \right] + \underbrace{\PP \left[ \max_{k \leq n} X_k \cdot \ell > \frac{2}{3} \lambda n^{\frac{1}{2} + \alpha} , \text{ } X_n \cdot \ell < \frac{1}{3} \lambda n^{\frac{1}{2} + \alpha-\beta} \right]}_{(*)}
\\
&\leq  \PP\left[|\mathcal{R}^X _n | < n^{\frac{1}{2} + \alpha}\right] + \PP\left[ \min_{k \leq n} X_k \cdot \ell < - n^{\frac{1}{2} + \alpha} \right] + (*)\,.
\end{split}
\end{equation*}
From Proposition~\ref{prop41}  we have that $\PP\left[|\mathcal{R}^X _n | < n^{\frac{1}{2} + \alpha}\right]\leq e^{- \gamma_1 n^{\gamma_2}}$, while by Azuma's inequality (for supermartingales), we  have that
\begin{equation*} 
\PP \left[ \min_{k \leq n} X_k \cdot \ell < - n^{\frac{1}{2} + \alpha} \right] \leq n \exp\left(-C_2 n^{2 \alpha}\right) \,, 
\end{equation*}
for $C_2 = 1/2K^2$. Regarding the term $(*)$, corresponding to the probability that the random walk exits, before time $n$,  the strip $H$ from the top and at time $n$ is in the strip $H$, we show that, there exists $C_1 > 0$ such that 
 \begin{equation} \label{max42_pnn0}
 \PP \left[ \max_{k \leq n} X_k \cdot \ell > \frac{2}{3} \lambda n^{\frac{1}{2} + \alpha} , \ X_n \cdot \ell < \frac{1}{3} \lambda n^{\frac{1}{2} + \alpha -\beta}  \right] \leq n e^{-C_1 n^{2 \alpha}}\, .
 \end{equation}
Note that 
\begin{align}
\label{eq:factorLAMBDA}
\left\{\max_{k \leq n} X_k \cdot \ell > \frac{2}{3} \lambda n^{\frac{1}{2} + \alpha} ,  X_n \cdot \ell < \frac{1}{3} \lambda n^{\frac{1}{2} + \alpha - \beta}\right\} \!\!
\subset  \bigcup_{k=1}^n \!\left\{(X_n  - X_k) \cdot \ell < \!\left(\frac{1}{3n^{\beta}} - \frac{2}{3} \right) \! \lambda n^{\frac{1}{2} + \alpha}\right\}\, ,
\end{align}
and by Azuma's inequality for supermartingales with increments bounded by $K$ (see \cite[Lemma 1]{wormald1995differential}), for every $k = \{1, \ldots ,n-1\}$ it holds that 
\begin{align*}
\PP &\left[ X_n \cdot \ell - X_k \cdot \ell < \left(  \frac{1}{3n^{\beta}} - \frac{2}{3} \right) \lambda n^{\frac{1}{2} + \alpha}  \right]  \leq \PP \left[ X_n \cdot \ell - X_k \cdot \ell <  -\frac{1}{3} \lambda n^{\frac{1}{2} + \alpha}  \right] 
\\
&\qquad \qquad \leq \exp \left(- \frac{\left(\frac{1}{3}\right)^{2} \lambda^2 n^{1+2\alpha}}{2(n-k)K^2} \right) \nonumber
 \leq \exp \left(- \frac{\lambda^2 n^{2\alpha}}{18 K^2} \right)\, . 
\end{align*}
Note that, since the value defining the top of the strip $H$ (i.e., $\frac{2}{3} \lambda n^{\frac{1}{2} + \alpha} $) depends linearly in $\lambda$, it was possible to apply Azuma's inequality for supermartingales for all values of $\lambda>0$. In fact,  it was possible to factorize $\lambda$ in \eqref{eq:factorLAMBDA}, thus obtaining  that $\left(  \frac{1}{3n^{\beta}} - \frac{2}{3} \right)<0$ for all $n$. 
Then \eqref{max42_pnn0} follows from the usual union bound with $C_1= \left(\frac{1}{3} \lambda \right)^2/2K^2$.  
For the second term in \eqref{prop42p1_pn}, we thus obtain 
\begin{equation} \label{parc1_pnn0}
\begin{split}
&\PP\left[ E \cap G^c \right]  
\leq e^{- \gamma_1 n^{\gamma_2}} + 
n e^{-C_2 n^{2 \alpha}} + 
n e^{-C_1 n^{2 \alpha}}\;. 
\end{split}
\end{equation}

We now address the first term in \eqref{prop42p1_pn}. 
Using~\eqref{cond44_pnn0}
 on $G$ we have at least $|\mathcal{R}^X _n | - \frac{1}{3}n^{\frac{1}{2} + \alpha} \geq \frac{2}{3}n^{\frac{1}{2} + \alpha}$ sites visited on the excitation-allowing $A$. Therefore,
if we define  $D_k := \mathbb{E} [X_{k+1} - X_k | \mathcal{F}_k] $, by Condition~\ref{condição2},  on the event $G$  we have that, almost surely 
\begin{equation}\label{eq:m_W}
\begin{split}
\Big( \sum_{k=0}^{n-1} D_k \Big) \cdot \ell & \geq \frac{2}{3}n^{\frac{1}{2}+ \alpha} \frac{\lambda}{(n_0 + n)^{\beta}}
\geq \frac{2\lambda}{3}  \frac{n^{\frac{1}{2}+ \alpha}}{(n + n)^{\beta}} > \frac{2^{1-\beta}\lambda}{3}n^{\frac{1}{2}+ \alpha - \beta}\,. 
\end{split}
\end{equation}  
Denoting 
$F=\left\{ X_n \cdot \ell < \frac{1}{3} \lambda n^{\frac{1}{2}+ \alpha-\beta}\right\}  \cap G  $ and  using ~\eqref{eq:m_W}, we obtain 
\begin{align*}
\PP\left[ F \right] & \leq  \PP\Big[ X_n \cdot \ell - \Big( \sum_{k=1}^{n-1} D_k \Big) \cdot \ell < \frac{1}{3} \lambda n^{\frac{1}{2}+ \alpha-\beta} - \frac{2^{1-\beta}\lambda}{3}n^{\frac{1}{2}+ \alpha - \beta} \Big]
\\
& \leq  \PP\Big[ X_n \cdot \ell - \Big( \sum_{k=1}^{n-1} D_k \Big) \cdot \ell <  \lambda n^{\frac{1}{2} + \alpha-\beta} \Big(  \frac{1}{3} - \frac{2^{1-\beta}}{3} \Big) \Big]\,.
\end{align*}
Setting $C'_4 := - (1/3 - 2^{1-\beta}/3) > 0$,
using that  $Y_n := X_n - \sum_{k=0}^{n-1} D_k$  is a martingale with bounded increments and  applying  Azuma's inequality  (see \cite[Theorem 2.19]{chung2006complex})   
\begin{equation}
\label{parcela4_pnn0}    
\begin{split}
\PP[F] & \leq  \PP\Big[ \Big( X_n  -  \sum_{k=1}^{n-1} D_k \Big) \cdot \ell <\! - C'_4  \lambda n^{\frac{1}{2} + \alpha-\beta} \Big] 
\!\leq  2 \!\exp \Big(\!\!- \!\frac{(C'_4)^{2} \lambda^2 n^{2 (\alpha-\beta)}}{2K^2}\Big) = 2 e^{-C'_5 n^{2(\alpha-\beta)}}\,, 
\end{split}
\end{equation}
where $C'_5  = (C'_4 \lambda)^2/2K^2$.
Then,  from \eqref{parc1_pnn0} and \eqref{parcela4_pnn0} we obtain  that 
\begin{equation*}
\label{eqfinal}
\begin{split}
\PP & \left[ X_n \cdot \ell < \frac{1}{3} \lambda n^{\frac{1}{2}+ \alpha-\beta} \right] \leq 
 e^{- \gamma_1 n^{\gamma_2}} + 
n e^{-C_2 n^{2 \alpha}} + 
ne^{-C_1 n^{2 \alpha}} +  2 e^{- C'_5   n^{2 (\alpha-\beta)}} 
\leq 5ne^{-\vartheta_1 n^{\vartheta_2}}\,,
\end{split}
\end{equation*}
where $\vartheta_2 = \min\left\{ \gamma_2, 2(\alpha-\beta) \right\}$ and
$$
\vartheta_1 = \min\left\{ \gamma_1, \frac{1}{2K^2}, \frac{\lambda^{2}}{18K^2}, \frac{((1/3 -2^{1-\beta}/3) \lambda)^2}{2K^2} \right\}\,.
$$
 \qed

\bibliographystyle{abbrv}
\bibliography{referencias}
 
 \end{document}